\documentclass[letterpaper, 10 pt, conference]{ieeeconf}
\usepackage[T1]{fontenc}% optional T1 font encoding
\IEEEoverridecommandlockouts                          
\usepackage{amsmath} % assumes amsmath package installed
\usepackage{amssymb}  % assumes amsmath package installed
\usepackage{xcolor}
\usepackage{caption}
\usepackage{graphicx}
\usepackage{mathtools}
\usepackage{caption}
\usepackage{floatrow}
\usepackage{tikz}
\usepackage{tikz-cd}

\usepackage{algorithmicx}
    \usepackage[ruled]{algorithm}
    \usepackage{algpseudocode}

\usepackage{pgfplots}
\usepackage{xcolor}
\usetikzlibrary{matrix,arrows,calc,positioning,shapes,decorations.pathreplacing}
\usepackage{graphicx}

\usepackage{amsmath} % assumes amsmath package installed

\usepackage{enumerate}
\usepackage[all,tips]{xy}
\SelectTips{cm}{11}

\newtheorem{theorem}{Theorem}
\newtheorem{corollary}{Corollary}
\newtheorem{definition}{Definition}

\newtheorem{proposition}{Proposition}
\newtheorem{lemma}{Lemma}
\newtheorem{example}{Example}

\newcommand{\R}{\mathbb{R}}

\begin{document}

\title{
Higher-order retraction maps and construction of numerical methods for optimal control of mechanical systems}

\author{Alexandre Anahory Simoes, Mar\'ia Barbero Liñ\'an, Leonardo Colombo, David Mart\'in de Diego. 
\thanks{A. Anahory Simoes (alexandre.anahory@ie.edu) is with the School of Science and Technology, IE University, Spain.}
\thanks{ M. Barbero  Li\~n\'an (m.barbero@upm.es) is with Departamento de Matem\'atica Aplicada, Universidad Polit\'ecnica de Madrid, Av. Juan de Herrera 4, 28040 Madrid, Spain.}
\thanks{L. Colombo (leonardo.colombo@car.upm-csic.es) is with Centre for Automation and Robotics (CSIC-UPM), Ctra. M300 Campo Real, Km 0,200, Arganda
del Rey - 28500 Madrid, Spain.} \thanks{David Mart\'in de Diego (david.martin@icmat.es) is with the Institute of Mathematical Sciences (CSIC-UAM-UCM-UC3M). Calle Nicol\'as Cabrera 13-15, Cantoblanco, 28049, Madrid, Spain.} \thanks{The authors acknowledge financial support from Grant PID2019-106715GB-C21 funded by MCIN/AEI/ 10.13039/501100011033.}}%

\maketitle

\begin{abstract}
    Retractions maps are used to define a discretization of the tangent bundle of the configuration manifold as two copies of the configuration manifold where the dynamics take place. Such discretization maps can be conveniently lifted to a higher-order tangent bundle to construct geometric integrators for the higher-order Euler-Lagrange equations. Given a cost function, an optimal control problem for fully actuated mechanical systems can be understood as a higher-order variational problem. In this paper we introduce the notion of a higher-order discretization map associated with a retraction map to construct geometric integrators for the optimal control of mechanical systems. In particular, we study applications to path planning for obstacle avoidance of a planar rigid body.
\end{abstract}

%\begin{IEEEkeywords} Symmetry reduction, multi-agent systems, optimal control, collision and obstacle avoidance. \end{IEEEkeywords}

\section{Introduction}
In this paper, we consider fully-actuated optimal control problems as  higher-order variational problems (see \cite{B} and \cite{BlochCrouch2}). Such problems are defined on the $k^{th}$-order tangent bundle $T^{(k)}Q$ of a differentiable manifold $Q$ (see \cite{BookLeon}). For a higher-order Lagrangian function $L:T^{(k)}Q\to\mathbb{R}$ and local coordinates $(q,\dot{q},\ldots, q^{(k)})$ on $T^{(k)}Q$  the  higher-order variational problems are given by
$$\min_{q(\cdot)}\int_{0}^{T}L(q(t),\dot{q}(t),\ldots,q^{(k)}(t))dt,$$
subject to the boundary conditions $q^{(j)}(0)=q_{0}^{j}$, $q^{(j)}(T)=q_{T}^{j}$ for $0\leq j\leq k-1$, where $q^{(j)}(t)=\frac{d^{j}}{dt^{j}}q(t)$.

The relationship between higher-order variational problems and optimal control problems of fully-actuated mechanical systems comes from the fact that Euler-Lagrange equations are represented by a second-order Newtonian system and  fully-actuated mechanical control systems have the form $F(q,\dot{q},\ddot{q})=u$, where $u$ are the control inputs, as many as the dimension of the configuration manifold $Q$. If $C$ is a cost function of an optimal control problem given by
$$\min_{(q(\cdot),u(\cdot))}\int_{0}^{T}C(q,\dot{q},u)dt,$$ it can be rewritten as a second-order variational problem replacing $u$ by the above expression.

The notion of retraction map is an essential tool in different research areas like optimization theory, numerical analysis and interpolation (see \cite{AbMaSeBookRetraction} and references therein). A retraction map plays the role of generalizing the linear-search methods in Euclidean spaces to general manifolds. On a manifold with nonzero curvature to move along the tangent line does not guarantee that the motion stays on the manifold. The retraction
map provides the tool to define the notion of moving in a direction of a tangent vector
while staying on the manifold. That is why retraction maps have been widely used to
construct numerical integrators of ordinary differential equations, since it allows us to
move from a point and a velocity to one nearby point so that the differential equation
can be discretized. 

In \cite{21MBLDMdD} the classical notion of retraction map used to approximate geodesics is extended to the new notion of discretization maps, that is rigorously defined to become a powerful tool to construct geometric integrators. Using the geometry of the tangent and cotangent bundles, the authors were able to tangently and cotangent lift the map so that these lifts inherit the same properties as the original one and they continue to be discretization maps. In particular, the cotangent lift of a discretization map is a natural symplectomorphism, which plays a key role for constructing symplectic integrators. It was further applied in \cite{DM} to the construction of numerical methods for optimal control problems from a Hamiltonian perspective. 

Geometric integrators for optimal control problems seen as second-order variational problems were studied in \cite{leo} (see also \cite{leo2}, \cite{leo3}). The goal of this paper is to extend the notion of discretization map given in~\cite{21MBLDMdD} to higher-order tangent bundles and construct symplectic integrators for optimal control problems of only fully-actuated mechanical systems.

The paper is structured as follows. Section~\ref{sec2} introduces the necessary tools on differential geometry and the geometric formalism for the dynamics of mechanical systems. Section~\ref{sec3} describes optimal control problems as higher-order variational problems and the Lagrangian and Hamiltonian characterization of necessary conditions for optimality. In Section~\ref{sec4} we  introduce retraction maps and discretization maps as well as the cotangent lift of discretization maps which allows the construction of symplectic integrators. In Section~\ref{sec5} we define higher-order discretization maps and describe the construction of symplectic integrators for higher-order mechanical systems. We employ this construction in Section~\ref{sec6}
to construct geometric integrators for optimal control of mechanical systems. In particular, we study applications to path planning for obstacle avoidance of planar rigid bodies.% and geometric integrators on the sphere.

\section{Background on Geometric Mechanics}\label{sec2}

Let $Q$ be a $n$-dimensional differentiable configuration
manifold of a mechanical system with local coordinates $(q^A)$, $1\leq A\leq n$. Denote by $TQ$ the
tangent bundle. If $T_{q}Q$ denotes the tangent space of $Q$ at the point $q$, then $\displaystyle{TQ:=\cup_{q\in Q}T_{q}Q}$, with induced local coordinates $(q^A, \dot{q}^A)$.  There is a canonical projection $\tau_{Q}:TQ \rightarrow Q$, sending each vector $v_{q}$ to the corresponding base point $q$. Note that in coordinates $\tau_{Q}(q^{A},\dot{q}^{A})=q^{A}$. 

The vector space structure of $T_{q}Q$ makes possible to consider its dual space, $T^{*}_{q}Q$, to define the cotangent bundle as $\displaystyle{T^{*}Q:=\cup_{q\in Q}T^{*}_{q}Q},$ with local coordinates $(q^A,p_A)$.  There is a canonical projection $\pi_{Q}:T^{*}Q \rightarrow Q$, sending each momenta $p_{q}$ to the corresponding base point $q$. Note that in coordinates $\pi_{Q}(q^{A},p_{A})=q^{A}$.

Given a Lagrangian function $L:TQ\rightarrow \R$, the corresponding Euler-Lagrange
equations are
\begin{equation}\label{qwer}
\frac{d}{dt}\left(\frac{\partial L}{\partial\dot
q^A}\right)-\frac{\partial L}{\partial q^A}=0, \quad 1\leq A\leq n.
\end{equation}

Equations \eqref{qwer} determine a system of $n$ second-order
differential equations. If we assume that the Lagrangian is regular,
i.e., the ${(n\times n)}$-matrix $\left(\frac{\partial^{2} L}{\partial \dot q^A\partial \dot q^B}\right)$, $1\leq A, B\leq n$, is non-singular, the local existence and uniqueness of solutions are guaranteed for any given initial condition by employing the Implicit Function Theorem.

A Hamiltonian function $H:T^{*}Q\to\mathbb{R}$ is described by the total energy of a mechanical system and leads to Hamilton's equations on $T^{*}Q$, whose solutions are integral curves of the Hamiltonian vector field $X_H$ taking values in $T(T^{*}Q)$ associated with $H$. % as a solution to the equation $i_{X_H}\Omega_c=\hbox{d}H$. 
Locally, $X_H(q,p)=\left(\frac{\partial H}{\partial p},-\frac{\partial H}{\partial q}\right)$, that is,
\begin{equation}\label{hameq1}\dot{q}^{A}=\frac{\partial H}{\partial p_A},\quad\dot{p}_{A}=-\frac{\partial H}{\partial q^A},\quad 1\leq A\leq n.\end{equation} Equations~\eqref{hameq1} determine a set of $2n$ first order ordinary differential equations (see \cite{B}, for instance, for more details).

A one-form $\alpha$ on $Q$ is a map assigning to each point $q$ a cotangent vector on $q$, that is, $\alpha(q)\in T^{*}_qQ$. Cotangent vectors acts linearly on vector fields according to $\alpha(X) = \alpha_{i}X^{i}\in \mathbb{R}$ if $\alpha = \alpha_{i}dq^{i}$ and $X = X^{i} \frac{\partial}{\partial q^{i}}$. Analogously, a two-form or a $(0,2)$-tensor field is a bilinear map that acts on a pair of vector fields to produce a number. %{\color{red}{Do we need this??? and also to $(1,1)$-tensor fields which are linear maps that act on a vector field to produce a new vector field.
%}}

A symplectic form $\omega$ on a manifold $Q$ is a $(0,2)$-type tensor field that is skew-symmetric and non-degenerate, i.e., $\omega(X,Y)=-\omega(Y,X)$ for all vector fields $X$ and $Y$ and if $\omega(X,Y)=0$ for all vector fields $X$, then $Y\equiv 0$. 

The set of vector fields and the set of 1-forms on $Q$ are denoted by $\mathfrak{X}(Q)$ and $\Omega^{1}(Q)$, respectively. The symplectic form induces a linear isomorphism $\flat_{\omega}:\mathfrak{X}(Q)\rightarrow \Omega^{1}(Q)$, given by $\langle\flat_{\omega}(X),Y\rangle=\omega(X,Y)$ for any vector fields $X, Y$. The inverse of $\flat_{\omega}$ will be denoted by $\sharp_{\omega}$.

As described in~\cite{LiMarle}, the cotangent bundle $T^*Q$ of a  differentiable manifold $Q$ is equipped with a canonical exact symplectic structure $\omega_Q=-d\theta_Q$, where $\theta_Q$ is the canonical 1-form on $T^*Q$. In canonical bundle coordinates $(q^A, p_A)$ on $T^*Q$, % the projection reads as $\pi_Q(q^i, p_i)=(q^i)$, and 
$
\theta_Q= p_A\, \mathrm{d}q^A$ and
$\omega_Q= \mathrm{d}q^A\wedge \mathrm{d}p_A\; .
$
Hamilton's equations can be intrinsically rewritten as $
\imath_{X_H}\omega_Q\colon = \flat_\omega(X_H)=\mathrm{d}H\; .
$
Hamiltonian dynamics are characterized by the following two essential properties~\cite{hairer}:
\begin{itemize}
	\item  Preservation of energy by the Hamiltonian function:  $$0=\omega_Q(X_H, X_H)=dH(X_H)=X_H(H)\,.$$
	
	\item Preservation of the symplectic form: %that is, the Lie derivative of the 2-form $\omega_Q$ with respect to the Hamiltonian vector field vanishes: $L_{X_H}\omega_Q=0$. Equivalently, 
 If $\{\phi^t_{X_H}\}$ is the flow of $X_H$, then the pull-back of the differential form by the flow is preserved, $(\phi^t_{X_H})^*\omega_Q=\omega_Q$.

\end{itemize}

 Recall that a
pair $(Q,\omega_{Q})$ is called a symplectic manifold if $Q$ is a
differentiable manifold and  $\omega_Q$ is a symplectic 2-form. As a consequence, the restrictions of $\omega_Q$ to each $q\in Q$ makes the tangent space $T_{q}Q$ into a symplectic vector space.

\begin{definition}
Let $(Q_1,\omega_1)$ and
$(Q_2,\omega_{2})$ be two symplectic manifolds, let $\phi:Q_1\to Q_2$ be a smooth map. The
map $\phi$ is called symplectic if the symplectic forms are preserved:
$
\phi^*\omega_2=\omega_1$. Moreover, it is a \textit{symplectomorphism} if $\phi$ is a diffeomorphism and $\phi^{-1}$ is also
symplectic.
\end{definition}
%In
%particular, the cotangent lift of a diffeomorphism is always a
%symplectomorphism.
%Symplectic integrators~\cite{hairer} were designed to preserve the configuration manifold and preserve the canonical symplectic form. 

		Let $Q_1$ and $Q_2$ be $n$-dimensional manifolds and $F: Q_1\rightarrow Q_2$ be a smooth map. The
	{\it tangent lift} $TF: TQ_1\rightarrow TQ_2$  of $F$ is defined by $TF(v_q)=T_qF (v_q)\in T_{F(q)} Q_2\, \, \mbox{ where } v_q\in T_qQ_1$, and $T_qF$ is the tangent map of $F$ whose matrix is the Jacobian matrix of $F$ at $q\in Q_1$.

		As the tangent map $T_qF$ is linear, the dual map $T_{q}^*F\colon T^*_{F(q)}Q_2\rightarrow T^*_qQ_1$ is defined as follows:
  \[\langle(T^*_{q}F)(\alpha_2), v_{q}\rangle=\langle \alpha_2, T_{q}F(v_{q})\rangle\mbox{ for every } v_q\in T_qQ_1.\]
  Note that $(T^*_{q}F)(\alpha_2)\in T^*_qQ_1$. 
  
 \begin{definition}\label{def:colift}
Let $F: Q_1\rightarrow Q_2$ be a diffeomorphism. The vector bundle morphism 
	$\widehat{F}: T^*Q_1\rightarrow T^*Q_2$ defined by	$\widehat{F}=T^*F^{-1}$
	is called the cotangent lift of $F^{-1}$. 
	 \end{definition}
In other words, $\widehat{F}(\alpha_q)= 	T^*_{F(q)}F^{-1} (\alpha_q)$ where $\alpha_q\in T^*_q Q_1$. Obviously, $(T^*F^{-1})\circ (T^*F)={\rm Id}_{T^*Q_2}$.

\subsection{Higher-order tangent bundles}

The higher-order tangent bundle is essentially a generalization of the tangent space of the manifold $Q$ to higher-order derivatives, when one interprets tangent vectors as the velocity vector of some curve in $Q$. Analogously, an element of the $k$-th order tangent bundle can be defined as an equivalence relation identifying all curves that match up to $k$-th order derivative.

Let $c_1, c_2:\R\rightarrow Q$ be two curves on $Q$. Consider the equivalence relation $\sim_{k}$ at $0\in \R$ determined by the following two conditions:
\begin{enumerate}
    \item $c_1(0)=c_2(0)$;
    \item $c_1^{(i)}(0)=c_2^{(i)}(0)$ for all $1 \leqslant i \leqslant k$, where the notation $c^{(i)}$ represents the $i$-th derivative of $c$.
\end{enumerate}
In this case, we say that $c_1$ and $c_2$ are $\sim_{k}$-related at $0$. Moreover, the equivalence class of $c$ determined by $\sim_{k}$ is called the $k$-jet of $c$ and is represented by $j_{0}^{(k)}c$. The set of all $k$-jets at $0$ is denoted by $J_{0}^{(k)}(\R,Q)$ in some general contexts. But, from now on, it will be denoted by $T^{(k)}Q$, the $k$-th order bundle of $Q$. The $k$-th order bundle of $Q$ is a smooth manifold (see \cite{BookLeon}) and admits several fibrations: $\pi^{k}_{r}:T^{(k)}Q \rightarrow T^{(r)}Q$ mapping $j^{(k)}_{0}c \mapsto j^{(r)}_{0}c$ for $0\leqslant r< k$. Observe that for $r=1$, $T^{(1)}Q = TQ$ and for $r=0$, $T^{(0)}Q = Q$.

If $(q^{A})$ are local coordinates on the manifold $Q$, then the $k$-jet $j_{0}^{(k)}c$ is uniquely determined by the coordinates $(q^{A}, q^{{(0)}^A}, \dots, q^{{(k)}^A}),$ where
$$q^{A}=c^{A}(0), \quad  q^{{(r)}^A} = \frac{1}{r!}c^{{(r)}^A}(0), \quad 1\leqslant A \leqslant \dim Q.$$
In a sense, the local coordinates for $k$-jets are provided by the Taylor polynomial of $c$ at $0$.

Given a smooth map $F:Q_{1} \rightarrow Q_{2}$, we define $T^{(k)}F:T^{(k)}Q_{1} \rightarrow T^{(k)}Q_{2}$ by $T^{(k)}F (j_{0}^{(k)}c) = j_{0}^{(k)}(F\circ c)$, for some curve $c:\R\rightarrow Q_1$.
 
%	We want to use the results in (MdDBL), to produce geometric numerical methods approximating optimal solutions for this problem. But first we need to get this results generalized for higher-order tangent bundle as well as study the Hamiltonian version of optimality conditions.

	\section{Variational Formulation of Optimal Control Problems for Mechanical Systems}\label{sec3}

	There are some problems in which the functional to be minimized depends on higher-order derivatives of a curve. This is the case in interpolating problems \cite{BlochHussein}, \cite{SCC00}; in generation of trajectories for quadrotors \cite{MK}, \cite{MK2}, or in a generalization of least square problem on Riemannian manifolds \cite{MacSil:2010}.
 
 %if the solution of the functional is a Riemannian polynomial \cite{CroSil:95}, \cite{MargaridaThesis}, \cite{Noa:89}, \cite{Zefran}, interpolation \cite{BlochHussein}, \cite{SCC00} or the generation of trajectories for quadrotors \cite{MK}, \cite{MK2}. Let $(Q,g)$ be a Riemannian manifold and $\frac{D}{Dt}$ be the
%covariant derivative along curves associated with the Levi-Civita
%connection $\nabla$ for the metric $g$.
%The Riemannian polynomials  are defined as minimizers of the
%functional 
 %   $$\int_{0}^{T} \Big{\|} \frac{D^{k} }{Dt^{k}}q \Big{\|}^{2} \ dt.$$
%  These kind of curves have applications in 

	The goal of this paper is to use discretization maps obtained from the retraction maps to produce numerical algorithms for the solutions of optimal control problems for fully actuated mechanical systems.	The prototype problem in this paper is the optimization of the cost functional
	$$\mathcal{J}=\int_{0}^{T} ||u||^{2} \ dt$$
	subjected to controlled Euler-Lagrange equations describing the dynamics of standard mechanical systems, i.e., 
	$$\frac{d}{dt} \left(\frac{\partial L}{\partial \dot{q}}\right) - \frac{\partial L}{\partial q} = u.$$
	
	The cost functional may then be recast as the second-order functional
	$$\mathcal{J}=\int_{0}^{T} \left|\left|\frac{d}{dt} \left(\frac{\partial L}{\partial \dot{q}}\right) - \frac{\partial L}{\partial q} \right|\right|^{2} \ dt= \int_{0}^{T}  \mathcal{L}(q,\dot{q}, \ddot{q}) \ dt.$$
	%This functional is a second-order functional of the type:
	%$$\mathcal{J}=\int_{0}^{T}  \mathcal{L}(q,\dot{q}, \ddot{q}) \ dt.$$	
	Using the variational principle, necessary equations for a trajectory to be optimal are the second order Euler-Lagrange equations:
	$$\frac{d^{2}}{dt^{2}} \left(\frac{\partial \mathcal{L}}{\partial \ddot{q}}\right) - \frac{d}{dt} \left(\frac{\partial \mathcal{L}}{\partial \dot{q}}\right) + \frac{\partial \mathcal{L}}{\partial q} = 0\, .$$

We want to use the results in \cite{21MBLDMdD} (see also \cite{DM}) to produce geometric numerical methods for the optimal control problem under study. But first, we need to get these results generalized for the higher-order tangent bundles. as well as to study the Hamiltonian version of optimality conditions.

%	\subsection{Hamiltonian version}
	
%	**People have written things about this. Read it carefully.
	
	A second-order Lagrangian $\mathcal{L}$ can be associated with a Lagrangian energy $E_{\mathcal{L}}:T^{(3)}Q\rightarrow \mathbb{R}$ defined by
	$$E_{\mathcal{L}}(q,\dot{q},\ddot{q}, q^{(3)})=\dot{q}\hat{p}_{(0)} + \ddot{q}\hat{p}_{(1)} - L(q, \dot{q}, \ddot{q}),$$
	where $\hat{p}_{(0)}$ and $\hat{p}_{(1)}$ are the generalized momenta given by
	\begin{equation*}
		\begin{split}
			\hat{p}_{(0)} & = \frac{\partial L}{\partial \dot{q}} - \frac{d}{dt}\frac{\partial L}{\partial \ddot{q}}, \,\,\,
			\hat{p}_{(1)} = \frac{\partial L}{\partial \ddot{q}}\, .
		\end{split}
	\end{equation*}
	These momenta are conserved along solutions of the second-order Euler-Lagrange equations (see \cite{BookLeon} for instance).
	
	As usual the link between Lagrangian and Hamiltonian formalism is the corresponding Legendre transformation $\text{Leg}_{\mathcal{L}}:T^{(3)}Q \rightarrow T^{*}(TQ)$ given by
	\begin{equation*}
		\text{Leg}_{\mathcal{L}}(q,\dot{q},\ddot{q}, q^{(3)}) = (q, \dot{q}, \hat{p}_{(0)}, \hat{p}_{(1)})\, .
	\end{equation*}
	
	The associated Hamiltonian function $H:T^{*}(TQ) \rightarrow \mathbb{R}$ is given by
	\begin{equation*}
		H(q, \dot{q}, \hat{p}_{(0)}, \hat{p}_{(1)}) = E_{\mathcal{L}} \circ \text{Leg}_{\mathcal{L}}^{-1} (q, \dot{q}, \hat{p}_{(0)}, \hat{p}_{(1)}),
	\end{equation*} and the second-order Hamitlon equations are given by $$\dot{q}=\frac{\partial H}{\partial \hat{p}_{(0)}},\quad\ddot{q}=\frac{\partial H}{\partial \hat{p}_{(1)}},\quad \dot{\hat{p}}_{(0)}=-\frac{\partial H}{\partial q},\quad\dot{\hat{p}}_{(1)}=-\frac{\partial H}{\partial\dot{q}}.$$

\section{Discretization maps}\label{sec4}

The first notion of retraction that appears in the literature can be found in~\cite{1931Borsuk} from a topological viewpoint. Later on, the notion of retraction map as defined below is used to obtain Newton's method on Riemannian manifolds~\cite{1986Shub,2002Adler}.

\begin{definition}\label{def-RetractMap} A \textit{retraction map} on a manifold $Q$ is a smooth mapping $R$ from the tangent bundle $TQ$ onto $Q$. Let $R_q$ denote the restriction of $R$ to $T_qQ$, the following properties are satisfied:
\begin{enumerate}
	\item $R_q(0_q)=q$, where $0_q$ denotes the zero element of the vector space $T_qQ$.
	\item With the canonical identification $T_{0_q}T_qQ\simeq T_qQ$, $R_q$ satisfies \begin{equation}\label{eq-DefRetract-prop}
	{\rm D}R_q(0_q)=T_{0_q}R_q={\rm Id}_{T_qQ},
	\end{equation}
	where ${\rm Id}_{T_qQ}$ denotes the identity mapping on $T_qQ$.
\end{enumerate}
\end{definition}

The condition~\eqref{eq-DefRetract-prop} is known as \textit{local rigidity condition} since, given $\xi\in T_qQ$,  the curve $\gamma_\xi(t)=R_q(t\xi)$ has $\xi$ as tangent vector at $q$, i.e. $\dot{\gamma}_\xi(t)= \langle DR_q(t\xi), \xi\rangle\; \hbox{ and, in consequence}, \dot{\gamma}_\xi(0)= {\rm Id}_{T_qQ}(\xi)=\xi$.

A typical example of a retraction map is the exponential map, $\hbox{exp}$, on Riemannian manifolds given in~\cite[Chapter 3.2]{doCarmo}. Therefore, the image of $\xi$ through the exponential map is a point on the Riemannian manifold $(Q,g)$ obtained by moving along a geodesic a length equal to the norm of $\xi$ starting with the velocity $\xi/\|\xi\|$, that is, 
\[
\hbox{exp}_q(\xi)=\sigma (\|\xi\|)\; ,
\]
where $\sigma$ is the unit speed geodesic such that $\sigma(0)=q$ and $\dot{\sigma}(0)=\xi/\|\xi\|$. 
%\begin{example}
%If  $(Q, g)$ is a Riemannian manifold, then the exponential map $\hbox{exp}^g: U\subset TQ\rightarrow Q$ is a typical example of retraction map: 
%$
%\hbox{exp}^g_q(v_q)=\gamma_{v_q}(1), 
%$
%where $\gamma_{v_q}$ is the unique  Riemannian geodesic satisfying $\gamma_{v_q}(0)=q$ and $\gamma'_{v_q}(0)=v_q$ (see~\cite{doCarmo} for instance). 
%\end{example}
% Remember that the exponential map is a typical example of a retraction map. With all that in mind we are able to generalize the property of local rigidity in Definition~\ref{def-RetractMap} that allows a discretization of the tangent bundle of the configuration manifold opening a new path to construct numerical integrators.

Next, we define a generalization of the retraction map in Definition~\ref{def-RetractMap} that allows a discretization of the tangent bundle of the configuration manifold leading to the construction of numerical integrators as described in~\cite{21MBLDMdD}. Given a point and a velocity, we obtain two nearby points that are not necessarily equal to the initial base point. %As discussed in the sequel, numerical methods will be recovered from this new map.

\begin{definition} \label{def:DiscreteMap2} A map 
	$R_d\colon U\subset TQ\rightarrow Q\times Q$ given by \begin{equation*}
	R_d(q,v)=(R^1(q,v),R^2(q,v)),
	\end{equation*} 
	where  $U$ is an open neighborhood of the zero section $0_q$ of $TQ$, 
	defines a {\it discretization map on $Q$} if it satisfies 
	\begin{enumerate}
		\item $R_d(q,0)=(q,q)$,
		\item $T_{0_q}R^2_q-T_{0_q}R^1_q\colon T_{0_q}T_qQ\simeq T_qQ\rightarrow T_qQ$ is equal to the identity map on $T_qQ$ for any $q$ in $Q$, where $R^a_q$ denotes the restrictions of $R^a$, $a=1,2$, to $T_qQ$.
	\end{enumerate}
\end{definition}
Thus, the discretization map $R_d$ is a local diffeomorphism from some neighborhood of the zero section of $TQ$.

If $R^1(q,v)=q$, the two properties in Definition~\ref{def:DiscreteMap2} guarantee that the both properties in Definition~\ref{def-RetractMap} are satisfied by $R^2$. Thus, Definition~\ref{def:DiscreteMap2} generalizes Definition~\ref{def-RetractMap}. 

%\begin{example}\label{Ex:Sphere}
%		Given a Riemannian manifold $(Q,g)$ and the associated exponential map $\hbox{exp}_q: T_qQ\rightarrow Q$, the following map 
%		\begin{equation}\label{eq:Rdexp}
%		R_d(q,\xi)=\left(\hbox{exp}_q(-\xi/2), \hbox{exp}_q(\xi/2)\right)\, 
%		\end{equation}
%		is a discretization map because it satisfies the properties in Definition~\ref{def:DiscreteMap2}. An example of discretization maps that can be associated with the exponential map is,  for instance, on the sphere $S^2$ with the Riemannian metric induced by the restriction of the standard metric on ${\mathbb R }^3$. The exponential map is given by 
%		\begin{equation}\label{s2-expmap}
%		    \hbox{exp}_q(\xi)=\cos ( \|\xi\|) \, q+\sin  (\|\xi\|) \,  \frac{\xi}{\|\xi\|}, \qquad \xi\in T_q S^2\, .
%		\end{equation}
		
%		Thus we move along the greatest circle that are the geodesics on the sphere. Remember that $	\hbox{exp}_q(0_q)=q$ and the exponential map is a continuous map. Hence, we can define the following discretization map on $Q$:
%		\begin{align*}\label{eq:RdS2mid}
%		R_d (q, \xi)=\left(\cos \left( \frac{\|\xi\|}{2}\right) q-\sin  \left(\frac{\|\xi\|}{2}\right)  \frac{ \xi}{\|\xi\|},\right.\\ \left.\cos \left( \frac{\|\xi\|}{2}\right) q+\sin \left(\frac{\|\xi\|}{2}\right) \frac{ \xi}{\|\xi\|}\right)\, .
%		\end{align*}
%		Another option is to use as a retraction map on the sphere the projection 
%		$R_q(\xi)=  \frac{q+\xi}{\|q+\xi\|}$ that leads to the following discretization map: 
%		\[
%			R_d (q, \xi)=\left(\dfrac{q-\xi/2}{\|q-\xi/2\|}, \dfrac{q+\xi/2}{\|q+\xi/2\|}\right).
%		\] 		
%	\end{example}

\begin{example} The mid-point rule on an Euclidean vector space can be recovered from the following discretization map:
$R_d(q,v)=\left( q-\dfrac{v}{2}, q+\dfrac{v}{2}\right).$

%Examples of retraction maps on Euclidean vector spaces typically used in the literature for the construction of numerical methods (see \cite{IserlesBook} for instance), that can be used to define discretization maps are:
%\begin{itemize}
%	\item Explicit Euler method:  $R_d(q,v)=(q,q+v).$\\
%	\item Midpoint rule:  $R_d(q,v)=\left( q-\dfrac{v}{2}, q+\dfrac{v}{2}\right).$\\
%		\item $\theta$-methods: $R_d(q,v)=\left( q-\theta \, v, q+ (1-\theta)\, v\right),$ with $\theta\in [0,1]$.
%\end{itemize}
\end{example}

%In general in a Riemannian manifold
 %$(M,g)$ the   associated exponential map  $exp^g_x: T_xM\rightarrow M$ also leads to the following discretization map:
%		\begin{equation*}\label{weyl}
%	R_d(v_x)=\left(exp^g_x(-v_x/2), exp^g_x(v_x/2)\right)\, .
	%\end{equation*}
%In a sphere we can use the orthogonal projection as a discretization map 
%	\[
%	R_d (x, \xi)=\left(\frac{x-\xi/2}{\|x-\xi/2\|}, \frac{x+\xi/2}{\|x+\xi/2\|}\right).
%	\] 

\subsection{Cotangent lift of discretization maps}

As the Hamiltonian vector field takes value on $TT^*Q$, the discretization map must be on $T^*Q$, that is, 
$R^{T^*}_d: TT^*Q \rightarrow T^*Q\times T^*Q$. Such a map is obtained by cotangently lifting a discretization map $R_d\colon TQ\rightarrow Q\times Q$, so that the construction $R^{T^*}_d$ is a symplectomorphism. In order to do that, we need the following three symplectomorphisms (see \cite{21MBLDMdD} and \cite{DM} for more details): 
\begin{itemize}
\item The cotangent lift of the diffeomorphism $R_d\colon TQ\rightarrow Q\times Q$ as described in Definition~\ref{def:colift}. %defined by:
	%\begin{equation*} 
	%\hat{F}: T^*Q_1 \longrightarrow  T^*Q_2 %\mbox{ such that } 
%\hat{F}=(TF^{-1})^*.
%	\end{equation*}
	\item The canonical symplectomorphism:
	\begin{equation*} \alpha_Q\colon T^*TQ  \longrightarrow  TT^*Q  
	\end{equation*}  \mbox{ such that } $\alpha_Q(q,v,p_q,p_v)= (q,p_v,v,p_q).$

	\item  The symplectomorphism between $(T^*(Q\times Q), \omega_{Q\times Q})$     and   
	$(T^*Q\times T^*Q, \Omega_{12}:=pr_2^*\omega_Q-pr^*_1\omega_Q)$:
		\begin{equation*}
	\Phi:T^*Q\times T^*Q \longrightarrow T^*(Q\times Q)\; , \; 
		\end{equation*} given by $\Phi(q_0, p_0; q_1, p_1)=(q_0, q_1, -p_0, p_1).$
	\end{itemize}
Diagram in Fig.~\ref{Diag:RdT*} summarizes the construction procress from $R_d$ to $R_d^{T^*}$:
	\begin{figure} [htb!]
 \begin{equation*}
\xymatrix{ {{TT^*Q }} \ar[rr]^{{{R_d^{T^*}}}}\ar[d]_{\alpha_{Q}} && {{T^*Q\times T^*Q }}  \\ T^*TQ \ar[d]^{\pi_{TQ}}\ar[rr]^{	\widehat{R_d}}&& T^*(Q\times Q)\ar[u]_{\Phi^{-1}}\ar[d]^{\pi_{Q\times Q}}\\ TQ \ar[rr]^{R_d} && Q\times Q }
\end{equation*}
    \caption{Definition of the cotangent lift of a discretization.}\label{Diag:RdT*}
\end{figure}

\begin{proposition}\cite{21MBLDMdD}
	Let $R_d\colon TQ\rightarrow Q\times Q$ be a discretization map on $Q$. Then $${{R_d^{T^*}=\Phi^{-1}\circ \widehat{R_d}\circ \alpha_Q\colon TT^*Q\rightarrow T^*Q\times T^*Q}}$$ 
is a discretization map  on $T^*Q$.\label{Prop:RdT*}
\end{proposition}
\begin{corollary}\cite{21MBLDMdD} The discretization map
 ${{R_d^{T^*}}}=\Phi^{-1}\circ 	(TR_d^{-1})^* \circ \alpha_Q\colon T(T^*Q)\rightarrow T^*Q\times T^*Q$ is a symplectomorphism between $(T(T^*Q), {\rm d}_T \omega_Q)$ and $(T^*Q\times T^*Q, \Omega_{12})$.
\end{corollary}
\begin{example}\label{example3} On $Q={\mathbb R}^n$ the discretization map 
	$R_d(q,v)=\left(q-\frac{1}{2}v, q+\frac{1}{2}v\right)$ is cotangently lifted to
		$$R_d^{T^*}(q,p,\dot{q},\dot{p})=\left( q-\dfrac{1}{2}\,\dot{q}, p-\dfrac{\dot{p}}{2}; \; q+\dfrac{1}{2}\, \dot{q}, p+\dfrac{\dot{p}}{2}\right)\, .$$
\end{example}

\section{Higher-order discretization maps}\label{sec5}
	In \cite{21MBLDMdD}, the authors show how to lift a discretization map to the tangent and cotangent bundles.	Next, we are going to see how to lift a discretization map to a one on a higher-order tangent bundle.
	
	Let $R_{d}:TQ \rightarrow Q \times Q$ be a discretization map on $Q$, then we can lift it to the map
	$$T^{(k)} {R}_{d}:T^{(k)}(TQ)\rightarrow T^{(k)}Q \times T^{(k)} Q,$$
    defined by $T^{(k)}R_{d}(j_{0}^{(k)}\gamma) = j_{0}^{(k)}(R_{d}\circ \gamma)$ for $\gamma:I\rightarrow TQ$.
 
	Consider the natural equivalence $\Phi^{(k)}:T(T^{(k)}Q) \rightarrow T^{(k)}(TQ)$ defined using the following construction (see \cite{Kolar} or \cite[Sec. V]{cant}): for each $X\in T(T^{(k)}Q)$ there exists a curve $c:\R\rightarrow Q$ such that $X= j_{0}^{(1)}(j_{0}^{(k)}c )$. Then, we have that 
    $$\Phi^{(k)}(X) = j_{0}^{(k)}(j_{0}^{(1)}c ).$$
    The identification between the higher-order tangent bundles $T^{(k)}(TQ) \cong T(T^{(k)}Q)$ allows to define the map
	$R_{d}^{(k)}:T(T^{(k)}Q) \rightarrow T^{(k)}Q \times T^{(k)} Q$ given by $R_{d}^{(k)} = T^{(k)}  {R}_{d} \circ \Phi^{(k)} $.
    
 The following lemma will be useful in the proof of the Theorem below.
 \begin{lemma}\label{lemma}
    Let $F:M\rightarrow N$ be a smooth map and $\gamma_{t}:\R \rightarrow M$ a smooth family of maps, i.e., $\gamma:\R^{2}\rightarrow M$ defined by $\gamma(t,s)=\gamma_{t}(s)$ is a smooth map.
    Then, 
    \begin{equation*}
        \left. \frac{d}{dt} \right|_{t=0} j_{0}^{(k)}(F\circ \gamma_{t}) = (\Phi_{N}^{(k)})^{-1} j_{0}^{(k)}\left(  \left. \frac{d}{dt} \right|_{t=0} (F\circ \gamma_{t}) \right)
    \end{equation*}
    where $\Phi_{N}^{(k)}:T(T^{(k)}N) \rightarrow T^{(k)}(TN)$ is the canonical identification.
\end{lemma}

\begin{proof}
    As 
    $$\left. \frac{d}{dt} \right|_{t=0} j_{0}^{(k)}(F\circ \gamma_{t}) = j_{0}^{(1)} (j_{0}^{(k)}(F\circ \gamma_{t}))\, ,$$
    using the natural equivalence $\Phi^{(k)}_N$
    $$\left. \frac{d}{dt} \right|_{t=0} j_{0}^{(k)}(F\circ \gamma_{t}) = (\Phi_{N}^{(k)})^{-1}\left(j_{0}^{(k)} (j_{0}^{(1)}(F\circ \gamma_{t}))\right),$$
    the result follows.
\end{proof}

    Now, we can prove that the map $R_{d}^{(k)}$ is a discretization map on the higher-order bundle $T^{(k)}Q$.
	
	\begin{theorem}\label{k:tangent:lift}
		Let $R_{d}$ be a discretization map on $Q$, the lift to the higher-order tangent bundle $R_{d}^{(k)}:T(T^{(k)}Q) \rightarrow T^{(k)}Q \times T^{(k)} Q$ is a discretization map on $T^{(k)}Q$.
	\end{theorem}

    \begin{proof}
        Let $\Phi^{(k)}:T(T^{(k)}Q)\rightarrow T^{(k)}TQ$ be the diffeomorphism identifying both manifolds.

        First, we shall prove that given $z\in T^{(k)}Q$, we have that $R_{d}^{(k)}(0_{z})=(z,z)$, where $0_z$ is the zero section of the bundle $T(T^{(k)}Q)\rightarrow T^{(k)}Q$.

        The image of the zero section under $\Phi^{(k)}$ is the $k$-th jet lift of the zero section on $Q$, that is, $\Phi^{(k)}(0_{z}) = T^{(k)}\hat{0}(z)$, where $\hat{0}:Q\rightarrow TQ$,  as it is easily checked choosing natural coordinates on the higher-order tangent bundle. Thus, $R_{d}^{(k)}(0_{z}) = T^{(k)} {R}_{d}(j^{(k)}\hat{0}(z))$.

        Using the definition of the $k$-th jet lift
        $$T^{(k)} {R}_{d}(j^{(k)}\hat{0}(z)) = j_{0}^{(k)}(R_{d}\circ \hat{0})(z).$$
        In addition, since $R_{d}$ is a discretization map, we have that $R_{d}\circ \hat{0}={\rm Id}_{Q}\times {\rm Id}_{Q}$. Hence,
        \begin{equation*}
            \begin{split}
                R_{d}^{(k)}(0_{z}) & = T^{(k)}({\rm Id}_{Q}\times {\rm Id}_{Q})(z) \\
                & = ({\rm Id}_{T^{(k)}Q}\times {\rm Id}_{T^{(k)}Q}) (z)=(z,z).
            \end{split}
        \end{equation*}

        Next, let $R_{d,z}^{(k)}$ be the restriction of $R_{d}^{(k)}$ to the space $T_{z}(T^{(k)}Q)$, where $z\in T^{(k)}Q$. We can write $R_{d,z}^{(k)}=T^{(k)} {R}_{d} \circ \Phi_{z}^{(k)}$, where $\Phi_{z}^{(k)}$ is the restriction of $\Phi^{(k)}$ to $T_{z}(T^{(k)}Q)$.
        
        Moreover, if $R_{d,z}^{(k),a}$ denotes the composition of $R_{d,z}^{(k)}$ with the projection onto the ath-factor, $a=1,2$, then we will prove that $T_{0_{z}}R_{d,z}^{(k),2}(X_{z}) - T_{0_{z}}R_{d,z}^{(k),1}(X_{z}) = X_{z}$ for all $X_{z} \in T_{z}(T^{(k)}Q)$ and $z\in T^{(k)}Q$ under the identification $T_{0_{z}}T_{z}(T^{(k)}Q)\equiv T_{z}(T^{(k)}Q)$.           %Indeed, let $z\in T^{(k)}Q$ and $X_{z}\in T_{z}(T^{(k)}Q)$. 
        We have that
        {$$\left( R_{d,z}^{(k),2} - R_{d,z}^{(k),1}\right)(X_{z}) = \left( T^{(k)} {R}_{d}^{2} - T^{(k)} {R}_{d}^{1} \right)\circ \Phi_{z}^{(k)}(X_{z})\, .$$}
        Therefore,
        \begin{equation*}
            \begin{split}
                \left. \frac{d}{dt} \right|_{t=0} & \left( R_{d,z}^{(k),2} - R_{d,z}^{(k),1}\right)(tX_{z}) \\
                & = \left. \frac{d}{dt} \right|_{t=0} \left( T^{(k)} {R}_{d}^{2} - T^{(k)} {R}_{d}^{1} \right)\circ \Phi_{z}^{(k)}(tX_{z})\\
                & = \left. \frac{d}{dt} \right|_{t=0} T^{(k)} \left(  {R}_{d}^{2} - {R}_{d}^{1} \right)\circ \Phi_{z}^{(k)}(tX_{z})\\
                & = \left. \frac{d}{dt} \right|_{t=0} j_{0}^{(k)} \left(  ({R}_{d}^{2} - {R}_{d}^{1} )(tY(q)) \right),
            \end{split}
        \end{equation*}
        where $q=\pi^{k}_{0}(z)$, $Y(q)\in TQ$ is a curve such that $j_{0}^{(k)}Y = \Phi_{z}^{(k)}(X_{z})$. Moreover, if $\tilde{\pi}_{0}^{k}\colon T^{(k)}TQ \rightarrow TQ$ and using $\tau_{Q}\circ T\pi_{0}^{k} = \pi_{0}^{k}\circ \tau_{T^{(k)}Q}$,  then $Y(q)\in T_{q}Q$  and $Y = \tilde{\pi}_{0}^{k}(\Phi^{(k)}(X_{z}))= T\pi_{0}^{k}(X_{z})$ because the diagram in Fig.~$2$ is commutative.
        \begin{figure}[htb!]
            \centering
            \begin{tikzcd}
T(T^{(k)}Q) \arrow[rr, "\Phi^{k}"] \arrow[d, "\tau_{T^{(k)}Q}"'] \arrow[rrd, "T\pi_{0}^{k}"] &   & T^{(k)}TQ \arrow[d, "\tilde{\pi}_{0}^{k}"] \\
T^{(k)}Q \arrow[rd, "\pi_{0}^{k}"']                                                      &   & TQ \arrow[ld, "\tau_{Q}"]                  \\
                                                                                       & Q &                                           
\end{tikzcd}
    \caption{Commutative diagram.}
        \end{figure}

        Using Lemma \ref{lemma} we have that
        \begin{equation*}
            \begin{split}
                \left. \frac{d}{dt} \right|_{t=0} & \left( R_{d,z}^{(k),2} - R_{d,z}^{(k),1}\right)(tX_{z}) \\
                & =  (\Phi_{z}^{k})^{-1} j_{0}^{k}\left(  \left. \frac{d}{dt} \right|_{t=0}({R}_{d,q}^{2} - {R}_{d,q}^{1}) (tY(q))\right)\, .
            \end{split}
        \end{equation*}

        Using the second property from discretization maps, we obtain
        \begin{equation*}
            \begin{split}
                \left. \frac{d}{dt} \right|_{t=0} & \left( R_{d,z}^{(k),2} - R_{d,z}^{(k),1}\right)(tX_{z}) \\
                & =  (\Phi_{z}^{k})^{-1} j_{0}^{k}(Y) = X_{z},
            \end{split}
        \end{equation*}
        where the last step follows from the definition of $Y$.
        \end{proof}

	\begin{example}\label{Ex:midpoint}
		Consider the midpoint discretization map
		$$R_{d}(q, v) = \left( q-\frac{1}{2}v, q + \frac{1}{2}v \right). $$
		
		The lift of the midpoint to $T(TQ)$ is
		$$TR_{d} (q, v, \dot{q}, \dot{v}) = \left(q-\frac{1}{2} v, q + \frac{1}{2} v, \dot{q} -\frac{1}{2}\dot{v}, \dot{q} + \frac{1}{2}\dot{v}\right)$$
		and the second lift to $T^{(2)}(TQ)$ is
		$$T^{(2)}R_{d} (q, v; \dot{q}, \dot{v}; \ddot{q},\ddot{v}) =$$ $$\left(q-\frac{1}{2} v, q + \frac{1}{2} v, \dot{q} -\frac{1}{2}\dot{v}, \dot{q} + \frac{1}{2}\dot{v}, \ddot{q}-\frac{1}{2}\ddot{v}, \ddot{q} + \frac{1}{2}\ddot{v}\right)$$

  %{\color{red}{$$T^{(2)}R_{d} (q, v; \dot{q}, \dot{v}; \ddot{q},\ddot{v}) =$$ $$ \left(q-\frac{1}{2}\dot{q}, v-\frac{1}{2}\dot{v},\ddot{q}-\frac{1}{2}\ddot{v}, q + \frac{1}{2}\dot{q},  v + \frac{1}{2}\dot{v},  \ddot{q} + \frac{1}{2}\ddot{v}\right).$$}}
		Under the natural equivalence between higher-order tangent bundles, the map $R_{d}^{(2)}: T(T^{(2)}Q)\rightarrow T^{(2)}Q \times T^{(2)}Q$ is given by
		$$R_{d}^{(2)}(q,\dot{q},\ddot{q}; v,\dot{v}, \ddot{v})=$$ $$\left(q-\frac{1}{2}v, \dot{q}-\frac{1}{2}\dot{v},\ddot{q}-\frac{1}{2}\ddot{v}; q+\frac{1}{2}v, \dot{q}+\frac{1}{2}\dot{v},\ddot{q}+\frac{1}{2}\ddot{v}\right)\, .$$
		
		Then
		$R_{d}^{(2)}(q,\dot{q},\ddot{q}; 0,0,0)=\left(q,\dot{q},\ddot{q}; q,\dot{q},\ddot{q} \right)$,
		
        $$T_{0_{(q,\dot{q},\ddot{q})}} R_{d, (q,\dot{q},\ddot{q})}^{(2),2} = \begin{bmatrix}
            1/2 & 0 & 0 \\
            0 & 1/2 & 0 \\
            0 & 0 & 1/2
        \end{bmatrix}\, ,$$
        $$T_{0_{(q,\dot{q},\ddot{q})}} R_{d, (q,\dot{q},\ddot{q})}^{(2),1} = \begin{bmatrix}
            -1/2 & 0 & 0 \\
            0 & -1/2 & 0 \\
            0 & 0 & -1/2
        \end{bmatrix}\, .$$
		Therefore, $T_{0_{(q,\dot{q},\ddot{q})}}R_{d, (q,\dot{q},\ddot{q})}^{(2),2} - T_{0_{(q,\dot{q},\ddot{q})}}R_{d, (q,\dot{q},\ddot{q})}^{(2),1} = Id,$  and $R_{d}^{(2)}$ is a discretization map under the suitable identifications.
	\end{example}

 \begin{example}
     Consider the initial point discretization map on the sphere $R_d:T\mathbb{S}^{2} \rightarrow \mathbb{S}^{2} \times \mathbb{S}^{2}$
     \begin{equation*}
         R_d (q, \xi)=\left(q, \frac{q+\xi}{\|q+\xi\|}\right).
     \end{equation*}
     The lift to $T(T\mathbb{S}^{2})$ is the map $TR_d\colon T(T\mathbb{S}^2)\rightarrow T\mathbb{S}^2\times T\mathbb{S}^2$:
		$$TR_{d} (q, \xi, \dot{q}, \dot{\xi}) = \left(q, \dot{q}, \frac{q+\xi}{\|q+\xi\|}, \frac{\dot{q}+\dot{\xi}}{\|q+\xi\|} - \frac{\xi\cdot\dot{\xi}(q+\xi)}{\|q+\xi\|^{3}}\right)$$
    and the second lift $T^{(2)}(T\mathbb{S}^{2})$ is
    \begin{equation*}
        \begin{split}
            T^{(2)}R_{d} & (q, \xi, \dot{q}, \dot{\xi}, \ddot{q}, \ddot{\xi}) =\left(TR_{d}(q, \xi, \dot{q}, \dot{\xi}) , \ddot{q}, \frac{\ddot{q}+\ddot{\xi}}{\|q+\xi\|} \right. \\
            & - \frac{2\xi\cdot\dot{\xi}(\dot{q}+\dot{\xi}) + (\dot{\xi}\cdot\dot{\xi} + \xi \cdot \ddot{\xi})(q + \xi)}{\|q+\xi\|^{3}} \\
            & \left. + \frac{3\xi\cdot\dot{\xi}(q+\xi)}{\|q+\xi\|^{5}}\right).
        \end{split}
    \end{equation*}
    Composing with the natural identifications, we obtain a discretization map on $T^{(2)}\mathbb{S}^{2}$.
 \end{example}

 \begin{corollary} Let $R_{d}^{(k)}:T(T^{(k)}Q) \rightarrow T^{(k)}Q \times T^{(k)} Q$ be a higher-order discretization map on $T^{(k)}Q$. The cotangent lift $\left(R_{d}^{(k)}\right)^{T^*}:T(T^*(T^{(k)}Q)) \rightarrow T^*(T^{(k)}Q) \times T^*(T^{(k)}Q)$ is discretization map on $T^*(T^{(k)}Q)$. \label{corol:kRdT*}     
 \end{corollary}
 \begin{proof}
 In Fig.~\ref{Diag:RdT*} the discretization map at the bottom line can be replaced by the higher-order discretization $R^{(k)}_d$, whose existence has been proved in Theorem \ref{k:tangent:lift}. Such a map can be cotangently lifted as in Proposition~\ref{Prop:RdT*} to obtain the following discretization map $\left( R^{(k)}_d\right)^{T^*}\colon T\left(T^{*}(T^{(k)}Q)\right) \rightarrow T^{*}(T^{(k)}Q)\times T^{*}(T^{(k)}Q)$.
 \end{proof}
	
	\subsection{Geometric integrators on the higher-tangent bundle}
	
	The framework for the construction of geometric integrators is established by Proposition 5.1 in \cite{21MBLDMdD} wich reads:
	
	\begin{proposition}\label{Prop:symplectic:integrator}
		If $R_{d}$ is a discretization map on $Q$ and $H:T^{*}Q \rightarrow \R$ is a Hamiltonian function, then the equation
        \begin{align*}
            &(R_{d}^{T^*})^{-1}(q_{0},p_{0},q_{1},p_{1}) =\\ & \sharp_{\omega} \left( h dH \left[ \tau_{T^{*}Q} \circ (R_{d}^{T^*})^{-1}(q_{0},p_{0},q_{1},p_{1}) \right] \right)
        \end{align*}
        written for the cotangent lift of $R_{d}$ is a symplectic integrator.
	\end{proposition}

	The previous proposition adapts perfectly to our case since a higher-order Lagrangian on $T(T^{(k)}Q)$ has the corresponding Hamiltonian function on $T^{*}(T^{(k)}Q)$. The cotangent lift in Proposisition~\ref{Prop:RdT*} can be replaced by the higher-order cotangent lift in Corollary~\ref{corol:kRdT*}. As a result, we have constructed a symplectic integrator for the Hamiltonian version of the higher-order dynamics.

 %   \subsection{Higher-order geometric integrator on the sphere}

	\section{Application to Optimal Control problems}\label{sec6}
	
%	\subsection{ Simple Optimal control problem on $\mathbb{R}^{n}$}
	
%	Let us start with a simple situation. 
 Suppose that on $Q = \R^{n}$ we have the optimal control problem with cost functional
	$$\mathcal{J}=\int_{0}^{T} \frac{1}{2}||u||^{2} \ dt$$
	subjected to the Euler-Lagrange controlled dynamics
	$\ddot{q} = u.$
	
	This problem can be recasted as the second-order variational problem $\mathcal{J}=\int_{0}^{T} \frac{1}{2} ||\ddot{q}||^{2} \ dt$
with the second-order Lagrangian 
	$\mathcal{L} =\frac{1}{2} ||\ddot{q}||^{2}$ on $T^{(2)}Q$.
	Necessary conditions  for a trajectory to be optimal is to fulfill the second-order Euler-Lagrange equations, which in this case give the spline equations
	$q^{(4)} =0.$
	However, as described in Section~\ref{sec3}, the Hamiltonian for this second-order Lagrangian system is defined on $T^*TQ$:
	$$H(q,\dot{q},\hat{p}_{(0)},\hat{p}_{(1)}) = \frac{1}{2}\hat{p}_{(1)}^{2} + \hat{p}_{(0)}\dot{q}.$$
 A discretization map on $T^*TQ=T^*(T^{(1)}Q)$ is obtained by cotangently lifting a first-order discretization map on $TQ$, that corresponds with the tangent lift of a discretization map on $Q$ as defined in~\cite{21MBLDMdD}.

	As in Example~\ref{Ex:midpoint}, the midpoint discretization map
	$R_{d}(q, \dot{q}) = \left( q-\frac{1}{2}\dot{q}, q + \frac{1}{2}\dot{q} \right)$ is used to define $R^{(1)}_d\colon T(TQ)\rightarrow TQ\times TQ$.
	
	The first-order cotangent lift of the midpoint on $T^*(TQ)$ is a discretization map as proved in Corollary~\ref{corol:kRdT*}: 
 \begin{align*}\left( R^{(1)}_d\right)^{T^*}&(q,\dot{q},p_{0},p_{1};\dot{q},\ddot{q},\dot{p}_{(0)},\dot{p}_{(1)})= \\ &\left( q-\frac{1}{2}\,\dot{q},\dot{q}-\frac{1}{2}\,\ddot{q}, p_{(0)}-\frac{\dot{p}_{(0)}}{2}, p_{(1)}-\frac{\dot{p}_{(1)}}{2}; \right. \\ & \left. q+\frac{1}{2}\,\dot{q},\dot{q}+\frac{1}{2}\,\ddot{q}, p_{(0)}+\frac{\dot{p}_{(0)}}{2}, p_{(1)}+\frac{\dot{p}_{(1)}}{2}\right)\, .\end{align*}
	%$$R_{d}^{T} (q,\dot{q}, v, \dot{v}) = \left(q-%\frac{1}{2}\dot{q}, v-\frac{1}{2}\dot{v}, q + %\frac{1}{2}\dot{q}, v + \frac{1}{2}\dot{v}\right).$$
	
	As the Hamiltonian vector field takes values in  $T(T^*TQ)$, by Proposition~\ref{Prop:symplectic:integrator}, $\left(R_{d}^{(1)}\right)^{T^*}$ generates the following symplectic numerical scheme on $T^*TQ$: 
    \begin{equation*}
        \begin{split}
            & \frac{q_{1}-q_{0}}{h} =  \frac{\dot{q}_{1}+\dot{q}_{0}}{2},\quad
            \frac{\dot{q}_{1}-\dot{q}_{0}}{h} =  \frac{\hat{p}_{(1)1}+\hat{p}_{(1)0}}{2},\\
            & \frac{\hat{p}_{(0)1}-\hat{p}_{(0)0}}{h} =  0,\quad
             \frac{\hat{p}_{(1)1}-\hat{p}_{(1)0}}{h} = - \frac{\hat{p}_{(0)1}+\hat{p}_{(0)0}}{2}.
        \end{split}
    \end{equation*}
	Working out the expressions we obtain:
    \begin{equation*}
        \begin{split}
            & q_{1} = q_{0} + h\dot{q}_{0} + \frac{h^{2}}{2} \hat{p}_{(1)0} + \frac{h^{3}}{4} \hat{p}_{(0)0}, \\
            & \dot{q}_{1} =  \dot{q}_{0} + h \hat{p}_{(1)0} + \frac{h^{2}}{2} \hat{p}_{(0)0}, \\
            & \hat{p}_{(0)1} =  \hat{p}_{(0)0}, \quad
             \hat{p}_{(1)1} = \hat{p}_{(1)0} - h \hat{p}_{(0)0}.
        \end{split}
    \end{equation*}

	\subsection{Obstacle avoidance problem}
The following application is an optimal control problem with obstacle avoidance,	which is usually cast as a second-order variational problem of the form
	$$\int_{0}^{T} \left( \frac{1}{2} ||\ddot{q}||^{2} + V(q) \right) \ dt$$ (see \cite{BlCaCoCDC}, \cite{BlCaCoIJC}, \cite{goodman}).   
	The second order Lagrangian is in this case $\mathcal{L} =\frac{1}{2} ||\ddot{q}||^{2}+V(q)$ and the necessary equations for a trajectory to be optimal is the fulfilment of the Euler-Lagrange equations which in this case are the fouth-order system:
	$q^{(4)} + \nabla V(q)=0.$
   The Hamiltonian for this second-order Lagrangian system is
	$$H(q,\dot{q},\hat{p}_{(0)},\hat{p}_{(1)}) = \frac{1}{2}\hat{p}_{(1)}^{2} + \hat{p}_{(0)}\dot{q} - V(q).$$
    The associated symplectic method will be
    \begin{equation*}
        \begin{split}
            & \frac{q_{1}-q_{0}}{h} =  \frac{\dot{q}_{1}+\dot{q}_{0}}{2},\quad
             \frac{\dot{q}_{1}-\dot{q}_{0}}{h} =  \frac{\hat{p}_{(1)1}+\hat{p}_{(1)0}}{2},\\
            & \frac{\hat{p}_{(0)1}-\hat{p}_{(0)0}}{h} =  \nabla V\left( \frac{q_{1} + q_{0}}{2} \right),\\
            & \frac{\hat{p}_{(1)1}-\hat{p}_{(1)0}}{h} =  -\frac{\hat{p}_{(0)1}+\hat{p}_{(0)0}}{2}.
        \end{split}
    \end{equation*}

    \subsection{Obstacle avoidance for a planar rigid body}
    In this section suppose $Q=SE(2)$, that all maps are considered in a local coordinate chart with coordinates $q=(x,y,\theta)$ and that the artificial potential has the form
    $$\displaystyle{V(x,y,\theta)=\frac{\tau}{x^{2}+y^{2}-r^{2}}}.$$

    We simulate the optimal trajectory of the previous problem using $\tau = 1\times 10^{-20}$, $r=1$ and we simulate $N = 400$ steps with a step-size $h=0.01$. To measure the norm we use the euclidean metric.

    \begin{figure}[htb!]
        \centering
        \includegraphics[scale=0.5]{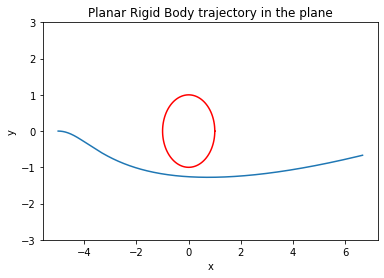}
        \caption{Trajectory of the optimal solution in the $xy$ plane.}
        \label{fig:RPB}
    \end{figure}

\section{Conclusions \& Further applications}

In this paper we have showed how to obtain discretization maps in higher-order tangent bundles by lifting discretization maps on the base manifold. Furthermore, we have showed some simple examples of higher-order discretization maps and simple applications to the construction of numerical integrators for optimal control problems. However, as we will describe below, the range of applications has still much to explore.
%\textcolor{red}{Melinger Kumar}    %\subsection{Optimal Control problem with Riemannian polynomials}
\subsection{Numerical methods for splines on the sphere}

		Given a Riemannian manifold $(Q,g)$ and the associated exponential map $\hbox{exp}_q: T_qQ\rightarrow Q$, the following map 		\begin{equation*}\label{eq:Rdexp}
		R_d(q,\xi)=\left(\hbox{exp}_q(-\xi/2), \hbox{exp}_q(\xi/2)\right)\, 
		\end{equation*}
		is a discretization map because it satisfies the properties in Definition~\ref{def:DiscreteMap2}. An example of discretization maps that can be associated with the exponential map is,  for instance, on the sphere $S^2$ with the Riemannian metric induced by the restriction of the standard metric on ${\mathbb R }^3$. The exponential map is given by 
		\begin{equation}\label{s2-expmap}		    \hbox{exp}_q(\xi)=\cos ( \|\xi\|) \, q+\sin  (\|\xi\|) \,  \frac{\xi}{\|\xi\|}, \qquad \xi\in T_q S^2\, .
		\end{equation}
		
		%Thus we move along the greatest circle that are the geodesics on the sphere. Remember that $	\hbox{exp}_q(0_q)=q$ and the exponential map is a continuous map. Hence, we can define the following discretization map on $Q$:
		%\begin{align*}\label{eq:RdS2mid}
		%R_d (q, \xi)=\left(\cos \left( \frac{\|\xi\|}{2}\right) q-\sin  \left(\frac{\|\xi\|}{2}\right)  \frac{ \xi}{\|\xi\|},\right.\\ \left.\cos \left( \frac{\|\xi\|}{2}\right) q+\sin \left(\frac{\|\xi\|}{2}\right) \frac{ \xi}{\|\xi\|}\right)\, .
		%\end{align*}
%	Another option is to use as a retraction map on the sphere the projection 
%		$R_q(\xi)=  \frac{q+\xi}{\|q+\xi\|}$ that leads to the following discretization map: 
%		\[
%			R_d (q, \xi)=\left(\dfrac{q-\xi/2}{\|q-\xi/2\|}, \dfrac{q+\xi/2}{\|q+\xi/2\|}\right).
%		\] 		

    Higher-order discretization maps can be used in the problem of finding higher-order Riemannian polynomials, defined in \cite{Noa:89}, \cite{CroSil:95}, \cite{MacSil:2010}, \cite{SCC00} as the critical curves of the higher-order functional 
    \begin{equation*}\label{riem:poly:func}
        \int_{0}^{T} \frac{1}{2}\langle \frac{D^{k} \gamma}{dt^{k}},\frac{D^{k} \gamma}{dt^{k}} \rangle \ dt,
    \end{equation*}
    where $\frac{D^{k} \gamma}{dt^{k}}$ denotes $k$-th covariant derivative.
    
    In future work, we will apply the previous construction to obtain higher-order geometric integrators to numerically obtain Riemannian polynomials. %Therefore, consider instead a functional defined by the projection onto $T\mathbb{S}^{2}$ of the third derivative of a curve $\gamma:\mathbb{R}\rightarrow \mathbb{S}^{2}\subseteq \R^{3}$:
   % $$\int_{0}^{T} \frac{1}{2} \left[ (q^{(3)})^{2} - (q^{(3)}\cdot q)^{2}\right] dt.$$

  %  When $k=3$, the problem above is associated with the third order Lagrangian function
   % $$\mathcal{L} = \frac{1}{2}(\dddot{q}^{2}) + \Gamma_{ij}^{k}$$

    %On the sphere we have the discretization map
    %$$R_{d}(q,\xi) = \left(q, \exp_{q}(\xi)\right),$$
    %where the exponential map is defined by \eqref{s2-expmap}.
	
\subsection{Discretization maps for systems on Lie groups}

Optimal control problems in Lie groups are extremely important because of the applications in robotics. Using the left-trivialized tangent bundle to have the identification $TG\approx G\times \mathfrak{g}$, the exponential map can be, for instance, used for defining a discretization map on the Lie group:
$$R_{d}(g,\xi) =\left(g \cdot \exp(-\xi/2),g \cdot \exp(\xi/2) \right).$$
In this scenario, higher-order lifts of $R_{d, g}$ are associated with higher-order derivatives of the map $R_{d, g}:\mathfrak{g}\rightarrow G\times G$. Then, we might generate geometric integrators for the problem \eqref{riem:poly:func} on Lie Groups.

\end{document}